\documentclass[reqno,12pt]{amsart}
\usepackage{amsfonts,amssymb,latexsym,epsf,amsthm,amsmath}
\usepackage{graphicx}
\usepackage{hyperref}

\makeatletter

\newtheorem{lemma}{Lemma}[section]
\newtheorem{theorem}[lemma]{Theorem}

\theoremstyle{definition}
\newtheorem{rem}[lemma]{Remark}
\newtheorem{definition}{Open Problem}

\numberwithin{equation}{section}

\makeatother

\def\trace{\text{trace}}
\def\diam{\text{diam}}
\def\iden{\text{I}}
\def\R{\mathbb{R}}

\def\div{\text{div}}

\title[Isoperimetric, Sobolev, and eigenvalue inequalities via ABP]
{Isoperimetric, Sobolev, and eigenvalue inequalities via the 
Alexandroff-Bakelman-Pucci method: a survey}

\author{Xavier Cabr\'e}

\address[X.C.]{
ICREA and Universitat Polit\`ecnica de Catalunya\\
Departament de Matem\`atica Aplicada~I\\
Diagonal 647. 08028 Barcelona. Spain.}
\email{xavier.cabre@upc.edu}

\thanks{The author was supported by grant MINECO MTM2011-27739-C04-01}

     
\begin{document}

\begin{abstract}
We present the proof of several inequalities using the technique introduced by 
Alexandroff, Bakelman, and Pucci to establish their ABP estimate.
First, we give a new and simple proof of a lower bound of Berestycki, Nirenberg, and Varadhan 
concerning the principal eigenvalue of an elliptic operator with bounded measurable coefficients.
The rest of the paper is a survey on the proofs of several isoperimetric and Sobolev inequalities
using the ABP technique. This includes new proofs of the classical isoperimetric inequality,
the Wulff isoperimetric inequality, and the Lions-Pacella isoperimetric inequality in convex cones.
For this last inequality, the new proof was recently found by the author, Xavier Ros-Oton, and
Joaquim Serra in a work where we also prove new Sobolev inequalities with weights
which came up studying an open question raised by Haim Brezis.
\end{abstract}

\date{}
\maketitle
\vspace {-.2in}
\begin{center}
{\em Dedicated to Haim Brezis, with great admiration}
\end{center}

\tableofcontents

\section{Introduction}

In this article we present the proof of several inequalities using the technique introduced by 
Alexandroff, Bakelman, and Pucci to establish their ABP estimate.
The Alexandroff-Bakelman-Pucci (or ABP) estimate
is an $L^\infty$ bound for solutions of the Dirichlet 
problem associated to second order uniformly elliptic operators
written in nondivergence form,
$$
Lu= a_{ij}(x) \partial_{ij} u +b_i (x) \partial_{i} u + c(x) u, 
$$
with bounded measurable coefficients in a domain 
$\Omega$ of $\R^n$. It asserts that if
$\Omega$ is bounded and $c\leq 0$ in $\Omega$ then,
for every function $u\in C^2(\Omega)\cap C(\overline\Omega)$,
\begin{equation}\label{ABP1}
\sup_\Omega u\le \sup_{\partial\Omega} u
+ C \, \diam (\Omega)\, \Vert Lu \Vert_{L^n(\Omega)} ,
\end{equation}
where $\diam (\Omega)$ denotes the diameter of $\Omega$, and $C$
is a constant depending only on the ellipticity constants of
$L$ and on the $L^n$-norms of the coefficients $b_i$
---see Remark~\ref{proofABP} below for its proof and Chapter~9 of \cite{GT} for more details.
The estimate was proven by the previous authors in the sixties using
a technique that in this paper we call ABP method.
Both the estimate and the method have applications in several areas. 

First, the ABP estimate is a basic tool in the 
regularity theory for fully nonlinear elliptic equations
$F(D^2u)=0$. The ABP method is also a key ingredient in Jensen's 
uniqueness result for viscosity solutions. For these
questions, see for instance \cite{CC}. Other applications were developed around 1994 by 
Berestycki, Nirenberg, and Varadhan \cite{BNV}, who established lower bounds on the principal eigenvalue of 
the operator $L-c(x)$ and, as a consequence, maximum principles in ``small'' domains. These maximum principles 
are very useful ---when combined with the moving planes method--- to establish symmetry of 
positive solutions of nonlinear problems (see \cite{BN,C2}).

In this paper we give a new and simple proof (unpublished before) of the lower bound of 
Berestycki, Nirenberg, and Varadhan~\cite{BNV} concerning the principal eigenvalue $\lambda_1=\lambda_1(L_0,\Omega)$ of 
the operator $L_0:=L-c(x)$, i.e.,
\begin{equation*}\label{L0}
L_0u= a_{ij}(x) \partial_{ij} u +b_i (x) \partial_{i} u.
\end{equation*}
The bound asserts that
\begin{equation}\label{lower}
\lambda_1 (L_0,\Omega) \ge \mu \vert\Omega\vert^{-2/n}
\end{equation}
for some positive constant $\mu$ depending only on the 
ellipticity constants of $L_0$, the $L^\infty$-norms
of the coefficients $b_i$, and an upper bound for $|\Omega|^{1/n}$. In particular,
if one has such upper bound for $|\Omega|$, then the constant $\mu$ is independent of $|\Omega|$. As a 
consequence, if $|\Omega|$ tends to zero then $\lambda_1(L_0,\Omega)$ tends to infinity, by \eqref{lower}.

In contrast with theirs, our proof uses
only the ABP method and does not require the  Krylov-Safonov Harnack inequality.
Our proof gives a slight improvement of this result
by showing that $\mu$ depends in fact on the $L^n$-norms of the coefficients $b_i$ instead of the $L^\infty$-norms. 
To prove this lower bound on $\lambda_1$, we apply the ABP method to the problem satisfied by
the logarithm of the principal eigenfunction of~$L_0$.

Note that the constant $\mu$ in the lower bound does not depend on any modulus of regularity for 
the coefficients of $L_0$.
This is why we say that it is a bound for operators with bounded measurable coefficients. This generality is
crucial for the applications to fully nonlinear elliptic equations.

When $L_0$ is in divergence form with bounded measurable coefficients, \eqref{lower} was proved by
Brezis and Lions~\cite{BL}. They established an estimate of the type \eqref{ABP1} with $L^n$ replaced by $L^\infty$.
When applied to the first eigenfunction, it gives \eqref{lower} for operators in divergence form.

An improvement of the ABP estimate \eqref{ABP1} in which $\diam (\Omega)$ is replaced by $|\Omega|^{1/n}$
was proved by the author in \cite{C1}; see also \cite{C2}.

When $L_0=\Delta$ is the Laplacian, \eqref{lower} with its best constant $\mu$ is the Faber-Krahn inequality,
and becomes an equality when $\Omega$ is a ball; see \cite{Ga}. Thus, among sets with same given volume, the ball has
the smallest first Dirichlet eigenvalue. In this respect we would like to raise the following:

\begin{definition}
When $L_0=\Delta$ is the Laplacian, can one prove the Faber-Krahn inequality 
(that is, inequality \eqref{lower} with best constant, achieved by balls) using an ABP method as described in
the following sections? 
\end{definition}

The rest of this paper is a survey in several isoperimetric inequalities proved using the ABP method.
We present first the proof of the classical isoperimetric inequality in $\R^n$ found by the author around 1996;
see~\cite{CSCM,CDCDS}. It uses the ABP technique applied to a linear Neumann problem for the Laplacian
---instead of applying the method to a Dirichlet problem as in the ABP estimate. 
It yields then the isoperimetric inequality with
best constant. In addition, the proof does not require the domain to be convex, and it shows easily that balls are the 
only smooth domains for which equality holds.

The proof using the ABP method can also be adapted to anisotropic perimeters. This gives a new proof of the Wulff isoperimetric inequality,
presented in Section~4.

The proof has also been recently extended by J. Serra and M. Teixid\'o \cite{ST}, in a very clever way, 
to domains in simply connected Cartan-Hadamard Riemannian manifolds of dimension two. 
These are manifolds with nonpositive sectional curvature.
In this way, they give a new proof that the Euclidean isoperimetric inequality (i.e., inequality \eqref{isop} below with the 
Euclidean constant $P(B_1)/|B_1|^{\frac{n-1}{n}}$) 
is also valid in such two-dimensional manifolds (with the same Euclidean constant on it). 
In higher dimensions (except for 3 and 4) 
this is an important conjecture which has been open for long time; see \cite{Dr}.

Finally, Section~5 concerns the recent paper \cite{CRS2}, by the author, X. Ros-Oton, and
J. Serra, where we established new isoperimetric and Sobolev inequalities with weights
in convex cones of $\R^n$. In particular we give a new poof of the Lions-Pacella isoperimetric inequality \cite{LP} 
in convex cones. Let us recall that the classical proofs of the Wulff and the Lions-Pacella isoperimetric inequalities
used the Brunn-Minkowski inequality \eqref{brunn}.

The result in \cite{CRS2} states that Euclidean balls centered at the origin solve the weighted isoperimetric problem 
in any open convex cone $\Sigma$ of $\R^n$ (with vertex at the origin) for the following class of weights.
Here, both perimeter and measure are computed with respect to the weight.
The weight $w$ must be nonnegative, continuous, positively homogeneous of degree $\alpha\geq 0$, and such that 
$w^{1/\alpha}$ is concave in the cone $\Sigma$ if $\alpha>0$.
This concavity condition is equivalent to a
natural curvature-dimension bound ---in fact, to the nonnegativeness of a
Bakry-\'Emery Ricci tensor in dimension $D=n+\alpha$. Except for the constant ones, all these weights are not radially 
symmetric but still balls centered at the origin are the isoperimetric sets.

Our proof uses the ABP method applied to a Neumann problem for the operator
$$
w^{-1}\div(w\nabla u)=\Delta u+\frac{\nabla w}{w}\cdot \nabla u.
$$

This result yields as a consequence the following Sobolev inequality.
If $D=n+\alpha$, $1\leq p<D$, and $p_*=\frac{pD}{D-p}$, then
\begin{equation}\label{Sob}
\left(\int_{\Sigma}|u|^{p_*}w(x)dx\right)^{1/p_*}\leq C_{w,p,n}
\left(\int_{\Sigma}|\nabla u|^pw(x)dx\right)^{1/p}
\end{equation}
for all smooth functions $u$ with compact support in $\R^n$ ---in particular,
not necessarily vanishing on $\partial\Sigma$. We can give the value of the best constant
$C_{w,p,n}$ since it is attained by certain radial functions; see \cite{CR2}.

Monomial weights,
\begin{equation}\label{mon}
w(x)=x_1^{A_1}\cdots x_n^{A_n}\qquad\text{in}\quad\Sigma=\{x\in\mathbb R^n\,:\, x_i>0
\textrm{ whenever }A_i>0\}
\end{equation}
(here $A_i\geq 0$), are an example of weights satisfying the above assumptions.
The Sobolev inequality \eqref{Sob} with the above monomial weights $w$ appeared naturally in the paper
\cite{CR}, by the author and X. Ros-Oton, while studying the following open question raised by Haim Brezis.

\begin{definition} \textbf{(Haim Brezis, 1996 \cite{B,BV})} Is the extremal solution of the problem
$-\Delta u=\lambda f(u)$ in a bounded smooth domain $\Omega\subset\R^n$, with zero Dirichlet boundary
conditions, always bounded if the dimension $n\leq 9$, and this for every positive, increasing, and convex nonlinearity~$f$? 
(see \cite{B,BV,CR} for more details).

A stronger statement is if the same conclusion holds for every stable solution of the Dirichlet problem for
$-\Delta u= f(u)$ in $\Omega$. It has been proved to be true in dimensions 2 and 3 by G. Nedev, in dimension 4 by the author, 
and in the radial case up to dimension 9 by the author and A. Capella; see the references in \cite{CSS}.
In \cite{C3}, we showed that these regularity results hold essentially for any nonnegative nonlinearity $f$.
\end{definition}

In \cite{CR} we studied this problem in convex domains with symmetry of double revolution, and we establish its
validity up to dimension $n\leq 7$.
If $\R^n=\R^m\times \R^k$, we say that a domain is of double revolution if it is invariant under 
rotations of the first $m$ variables and also under rotations of the last $k$ variables.
Stable solutions will depend only on the ``radial'' variables $s=\sqrt{x_1^2+\cdots +x_m^2}$ and
$t=\sqrt{x_{m+1}^2+\cdots +x_n^2}$. In these coordinates, the Lebesgue measure in $\R^n$
becomes $s^{m-1}t^{k-1}\, ds\, dt$. This is a monomial weight as in \eqref{mon}. In \cite{CR},
to prove regularity results we needed the above Sobolev inequalities with monomial weights,
even with nonintegers $A_i$ in \eqref{mon}.

\section{The principal eigenvalue for elliptic operators with bounded measurable coefficients}

The ABP estimate is the basic bound for subsolutions $u$ of
the Dirichlet problem
\begin{equation}
\left\{
\alignedat2
Lu &\ge f &\quad &\text{in } \Omega \\
u  &\le 0 &\quad &\text{on } \partial\Omega ,
\endalignedat
\right.
\label{dirich}
\end{equation}
where $L$ is an elliptic operator written in nondivergence form 
$$
Lu= a_{ij}(x) \partial_{ij} u +b_i (x) \partial_{i} u + c(x) u,
$$
in a domain $\Omega \subset \R^n$. We assume that $L$ 
is uniformly elliptic with bounded measurable coefficients, i.e.,
$b:=(b_1,\ldots ,b_n)\in L^\infty(\Omega)$,
$c\in L^\infty(\Omega)$ and
$$
c_0 \vert \xi\vert^2 \leq a_{ij} (x) \,\xi_i\,\xi_j
\leq C_0\vert \xi\vert ^2\qquad \forall
\xi\in {\R}^n \; \;\forall x\in \Omega  
$$
for some constants $0<c_0\le C_0$.
The ABP estimate states that, if 
$\Omega$ is bounded, $c\leq 0$ in $\Omega$, 
$u\in C^2(\Omega)\cap C(\overline\Omega)$ and \eqref{dirich}
holds, then
\begin{equation}
\sup_\Omega u\leq C \, \diam (\Omega)\, \Vert f\Vert_{L^n(\Omega)} ,
\label{abp}
\end{equation}
where $\diam (\Omega)$ denotes the diameter of $\Omega$ and $C$
is a constant depending only on  $n$, $c_0$, and 
$\Vert b\Vert_{L^n(\Omega)}$.

The proof of the ABP estimate is explained below in Remark~\ref{proofABP}, after
having presented in detail the ABP proof of the isoperimetric inequality.

In 1979, Krylov and Safonov used the ABP estimate and the 
Calder\'on-Zygmund cube decomposition to establish a deep
result: the Harnack inequality for second order uniformly
elliptic equations in nondivergence form with bounded 
measurable coefficients. This result allowed for the development 
of a regularity theory for fully nonlinear equations (see \cite{CC}).

Consider now the operator
$$
L_0u=(L-c(x))u=a_{ij}(x) \partial_{ij} u +b_i (x)\partial_i u ,
$$
and assume that $\Omega$ is 
a bounded smooth domain
and that the coefficients $a_{ij}$ are smooth in 
$\overline\Omega$.
In \cite{BNV} it is proved the existence 
of a unique eigenvalue $\lambda_1=\lambda_1 (L_0,\Omega)$ of $-L_0$ in $\Omega$ (the principal
eigenvalue) having a positive (smooth) eigenfunction $\varphi_1$
(the principal eigenfunction):
\begin{equation*}
\left\{
\alignedat3
L_0\varphi_1 &= -\lambda_1 \varphi_1 &\quad &\text{in } \Omega\\
\varphi_1   &=0 &\quad &\text{on }\partial \Omega \\
\varphi_1   &>0 &\quad &\text{in }\Omega .
\endalignedat
\right.
\label{preig}
\end{equation*}
In addition, $\lambda_1$ is a simple eigenvalue and satisfies
$\lambda_1 >0$. 

In Theorem~2.5 of \cite{BNV}, Berestycki, Nirenberg, and Varadhan used the Krylov-Safonov
theory to establish the lower bound
$\lambda_1 \ge \mu\vert\Omega\vert^{-2/n}$
for some positive constant $\mu$ depending only on 
$n$, $c_0$, $C_0$, and an upper bound on  $\vert\Omega\vert^{1/n}
\Vert b\Vert_{L^\infty(\Omega)}$.
We now give a simpler proof (unpublished before) of this lower bound 
using the ABP method. We do not need to use the Krylov-Safonov theory. Our proof improves slightly the bound
by showing that $\mu$ can be taken to depend on $\Vert b\Vert_{L^n(\Omega)}$ ---instead of
$|\Omega|^{1/n}\Vert b\Vert_{L^\infty(\Omega)}$. More precisely, we have the following:

\begin{theorem} 
If $\Omega$ is bounded, the principal eigenvalue $\lambda_1 (L_0,\Omega)$ of $L_0$ in $\Omega$ satisfies
\begin{equation*}
\lambda_1 (L_0,\Omega) \ge \mu \vert\Omega\vert^{-2/n},
\label{fklow}
\end{equation*}
where $\mu$ is a positive constant depending only on $n$,
$c_0$, $C_0$, and $\Vert b\Vert_{L^n(\Omega)}$.
\end{theorem}

\begin{proof}
Since $\varphi_1 >0$ in $\Omega$ we can consider 
the function 
$$
u=-\log\varphi_1 . 
$$
Using that $\nabla u=-\varphi_1^{-1}
\nabla\varphi_1$, we have that
\begin{equation}
\left\{
\alignedat2
a_{ij}\,\partial_{ij}u &= \lambda_1 - b_i\,\partial_i u
    +a_{ij}\,\partial_i u\,\partial_j u &\quad &\text{in } \Omega\\
u &=+\infty &\quad &\text{on }\partial \Omega .
\endalignedat
\right.
\label{prlog}
\end{equation}

We consider the lower contact set of $u$, defined by
\begin{equation*}
\Gamma_u =\{ x \in \Omega \ : \
u(y) \ge u(x) + \nabla u (x) \cdot (y-x)\  \text{ for all } y \in 
\overline \Omega \} .
\label{lcset0}
\end{equation*}
It is the set of points where the tangent hyperplane to the
graph of $u$ lies below $u$ in all $\overline \Omega$. 

For every $p\in\R^n$, the minimum  
$\min_{\overline\Omega}\,\{ u(y)-p\cdot y\}$ 
is achieved at an interior point of $\Omega$, since $u=+\infty$
on $\partial\Omega$ and $\Omega$ is bounded. At such a point $x$ in $\Omega$ of minimum
of the function $y \mapsto u(y)-p\cdot y$,
we have $x\in\Gamma_u$ and $p=\nabla u(x)$. It follows that
\begin{equation}
\R^n = \nabla u(\Gamma_u).
\label{gradmap0}
\end{equation}
It is interesting to visualize geometrically this proof by considering 
the graphs of the functions $p\cdot y + c$ for $c\in \R$. These are
parallel hyperplanes which lie, for $c$ close to $-\infty$, 
below the graph of $u$. We let~$c$ increase and consider the first~$c$ 
for which there is contact or ``touching'' at a point~$x$. 
It is clear that 
$x\not\in\partial\Omega$, since $u=+\infty$ on $\partial\Omega$.

Using \eqref{gradmap0}, we can apply the area formula to the map $p=\nabla u (x)$ for
$x\in\Gamma_u$ and, integrating in $\R^n$ a positive function
$g=g(\vert p\vert)$ to be chosen later, we obtain
\begin{equation}
\int_{\R^n} g(\vert p\vert) \, dp \le \int_{\Gamma_u} 
g(\vert\nabla u(x)\vert)\det D^2u(x) \,dx .
\label{arealog}
\end{equation}
Note that $D^2 u(x)$ is nonnegative definite at any point $x\in\Gamma_u$.

Next, we use the matrix inequality $\det (AB)\le\{\trace(AB)/n\}^n$,
which holds for every pair $A$ and $B$ of nonnegative symmetric
matrices. This is a simple extension of the arithmetic-geometric means inequality.
We apply it with $A=[a_{ij}(x)]$ and $B=D^2u(x)$ 
for $x\in\Gamma_u$. We also use that
$$
(a_{ij}\partial_{ij}u)^n\le 
C (\lambda_1^n+\vert b\vert^{n}\vert\nabla u\vert^n
+\vert\nabla u\vert^{2n}) \quad \text{in } \Gamma_u ,
$$
which follows from \eqref{prlog}. Here, and throughout the proof, 
$C$ will denote a positive
constant depending only on $n$,
$c_0$, $C_0$, and $\Vert b\Vert_{L^n(\Omega)}$. 
We deduce that
$$
\alignedat2
\det\, & D^2u \le c_0^{-n} \det ([a_{ij}] D^2u) 
\le c_0^{-n} \left( \frac{\trace ([a_{ij}] D^2u)}{n}\right)^n 
= (n c_0)^{-n} (a_{ij}\partial_{ij}u)^n \\
&
\le C (\lambda_1^n+\vert b\vert^{n}\vert\nabla u\vert^n
+\vert\nabla u\vert^{2n})\quad \text{in } \Gamma_u .
\endalignedat
$$
Therefore, choosing
$g(\vert p\vert)=(\lambda_1^n+\vert \Omega\vert^{-1}\vert p\vert^n
+\vert p\vert^{2n})^{-1}$ in \eqref{arealog}, we have
\begin{equation}
\alignedat2
\int_{\R^n} \frac{dp}{\lambda_1^n+\vert \Omega\vert^{-1}  \vert p\vert^n
+\vert p\vert^{2n}} & \le  \int_{\Gamma_u}
\frac{C (\lambda_1^n+\vert b\vert^{n}\vert\nabla u\vert^n
+\vert\nabla u\vert^{2n})}{\lambda_1^n+\vert \Omega\vert^{-1}  \vert \nabla u\vert^n
+\vert \nabla u\vert^{2n}}\,dx \\
& \le  
C \int_{\Gamma_u}  (1+\vert\Omega\vert \vert b\vert^n) \,dx \\
&\le C \left(1+\Vert b\Vert^n_{L^n(\Omega)}\right)\vert\Omega\vert 
\le C\vert\Omega\vert .
\endalignedat
\label{upint}
\end{equation}

On the other hand, using that 
$\lambda_1^n+\vert \Omega\vert^{-1}\vert p\vert^n
+\vert p\vert^{2n}\le\lambda_1^n+2\vert \Omega\vert^{-1}\vert
p\vert^n$ for $\vert p\vert \le \vert\Omega\vert^{-1/n}$,
we see that
\begin{equation}
\alignedat2
\int_{\R^n} \frac{dp}{\lambda_1^n+\vert \Omega\vert^{-1}\vert p\vert^n
+\vert p\vert^{2n}} &\ge 
\int_{B_{\vert\Omega\vert^{-1/n}}} 
\frac{dp}{\lambda_1^n+2\vert \Omega\vert^{-1}\vert p\vert^n}\\ 
&=c(n)\vert\Omega\vert\log\left( 1+\frac{2\vert
\Omega\vert^{-2}}{\lambda_1^n} \right).
\endalignedat
\label{loint}
\end{equation}
Combining \eqref{upint} and \eqref{loint}, we conclude 
$2\vert \Omega\vert^{-2}\lambda_1^{-n}\le C$,
which is the desired inequality.
\end{proof}

\section{The classical isoperimetric inequality}

In this section we present a proof of the 
classical isoperimetric problem for smooth domains of $\R^n$ which uses the ABP technique. 
It was found by the author in 1996 and published in \cite{CSCM,CDCDS}.
The proof establishes the following:

\begin{theorem}{\rm \textbf{(Isoperimetric inequality)}}
Let $\Omega$ be a bounded smooth domain of $\R^n$. Then
\begin{equation}
\frac{P(\Omega)}{\vert \Omega \vert^{\frac{n-1}{n}}}
\ge \frac{P(B_1)}{\vert B_1 \vert^{\frac{n-1}{n}}} 
\ ,
\label{isop}
\end{equation}
where $B_1$ is the unit ball of $\R^n$, $\vert \Omega \vert$ denotes
the measure of $\Omega$, and $P(\Omega)$ the perimeter of $\Omega$. Moreover, equality occurs 
in \eqref{isop} if and only if $\Omega$ is a ball of~$\R^n$.
\end{theorem}

\begin{proof} 
Let $u$ be a solution of the Neumann problem
\begin{equation}
\left\{
\alignedat2
\Delta u &= \frac{P(\Omega)}{\vert \Omega \vert} 
&\quad &\text{in } \Omega\\
\frac{\partial u}{\partial\nu} &=1 &\quad &\text{on }\partial \Omega ,
\endalignedat
\right.
\label{eqsem}
\end{equation}
where $\Delta$ denotes the Laplace operator and $\partial u /\partial\nu$
the exterior normal derivative of $u$ on $\partial \Omega$. The
constant $P(\Omega) /\vert \Omega \vert$ 
has been chosen so that the problem has a unique solution up to
an additive constant. 
For these classical facts, see Example 2 in Section 10.5 of \cite{H}, or the end of Section 6.7 of \cite{GT}.
In addition, we have that $u$ is smooth in~$\overline \Omega$. 

We consider the lower contact set of $u$, defined by
\begin{equation}
\Gamma_u =\{ x \in \Omega \ : \
u(y) \ge u(x) + \nabla u (x) \cdot (y-x)\  \text{ for all } y \in 
\overline \Omega \} .
\label{lcset}
\end{equation}
It is the set of points where the tangent hyperplane to the
graph of $u$ lies below $u$ in all $\overline \Omega$. We claim that
\begin{equation}
B_1 (0) \subset \nabla u (\Gamma_u) ,
\label{gradmap}
\end{equation}
where $B_1 (0)=B_1$ denotes the unit ball of $\R^n$ with center $0$.

To show \eqref{gradmap}, take any $p\in \R^n$ satisfying $\vert p \vert <1$. 
Let $x\in \overline \Omega$ be a point such that
$$
\min_{y\in \overline \Omega} \,\{ u(y)
-p\cdot y \} = u(x)-p\cdot x 
$$
(this is, up to a sign, the Legendre transform of $u$).
If $x\in \partial \Omega$ then the exterior normal derivative of 
$u(y)-p\cdot y$ at $x$ would be 
nonpositive and hence $(\partial u /\partial\nu)
(x) \le \vert p \vert <1$, a contradiction  with \eqref{eqsem}. It follows 
that $x\in \Omega$ and, therefore, that $x$ is an interior minimum of 
the function $u(y)-p\cdot y$. In particular, 
$p=\nabla u (x)$ and $x\in \Gamma_u$. Claim \eqref{gradmap} is now proved.
It is interesting to visualize geometrically the proof
of the claim, by considering 
the graphs of the functions $p\cdot y + c$ for $c\in \R$. These are
parallel hyperplanes which lie, for $c$ close to $-\infty$, 
below the graph of $u$.
We let~$c$ increase and consider the first~$c$ 
for which there is contact or ``touching'' at a point~$x$. 
It is clear geometrically that 
$x\not\in\partial\Omega$,
since $\vert p\vert <1$ and $\partial u /\partial\nu =1$ on 
$\partial\Omega$.

Next, from \eqref{gradmap} we deduce
\begin{equation}
\vert B_1\vert \le \vert \nabla u (\Gamma_u) \vert = 
\int_{\nabla u (\Gamma_u)} dp 
\le \int_{\Gamma_u} \det D^2u (x) \ dx .
\label{ineq}
\end{equation}
We have applied the area formula to the map $\nabla u : \Gamma_u
\rightarrow \R^n$, and we have used that its Jacobian,
$\det D^2u$, is nonnegative in $\Gamma_u$ by definition of this set. 

Finally, we use the arithmetic-geometric means inequality
applied to the eigenvalues of $D^2u(x)$ 
(which are nonnegative numbers for $x\in \Gamma_u$). We obtain
\begin{equation}
\det D^2u \le \left( \frac{\Delta u}{n} \right)^n \quad \text{in }
\Gamma_u .
\label{means}
\end{equation}
This, combined with \eqref{ineq} and 
$\Delta u \equiv P(\Omega) / \vert \Omega \vert$,
gives
\begin{equation}
\vert B_1 \vert \le \left( \frac{P(\Omega)}
{n \vert \Omega \vert} \right)^n \vert \Gamma_u \vert  
\le \left( \frac{P(\Omega)}
{n \vert \Omega \vert} \right)^n \vert \Omega \vert  .
\label{contact}
\end{equation}
Since $P( B_1) = n \vert B_1\vert$, we 
conclude the isoperimetric inequality
\begin{equation}
\frac{P( B_1)}{\vert B_1 \vert^{\frac{n-1}{n}}}= 
n\vert B_1 \vert^{\frac{1}{n}} \le
\frac{P(\Omega)}{\vert \Omega \vert^{\frac{n-1}{n}}}.
\label{isopfin}
\end{equation}

Note that when $\Omega=B_1$ then $u(x)=\vert x\vert^2/2$ and,
in particular, all the eigenvalues of $D^2u(x)$ are equal. 
Therefore, it is clear that \eqref{gradmap} and \eqref{means} are 
equalities when $\Omega=B_1$. This explains why the proof
gives the isoperimetric inequality with best constant.

The previous proof can also be used to show that balls are the only smooth
domains for which equality occurs in the isoperimetric inequality. Indeed,
if \eqref{isopfin} is an equality then all the inequalities in \eqref{ineq},
\eqref{means} and \eqref{contact} are also equalities. In particular, we have
$\vert \Gamma_u \vert = \vert \Omega\vert$. Since $\Gamma_u \subset \Omega$,
$\Omega$ is an open set, and $\Gamma_u$ is closed relatively to $\Omega$,
we deduce that $\Gamma_u = \Omega$.

Recall that the geometric and arithmetic means of $n$ nonnegative numbers
are equal if and only if these $n$ numbers are all equal. Hence, the equality 
in \eqref{means} 
and the fact that $\Delta u$ is constant in $\Omega$ give that 
$D^2 u = a\iden$ in all $\Gamma_u = \Omega$, where $\iden$ 
is the identity matrix and
$a=P(\Omega) /(n\vert\Omega\vert)$ 
is a positive constant. Let $x_0 \in \Omega$ be any given point.
Integrating $D^2u=a\iden$ on segments from $x_0$,
we deduce that
$$
u(x)=u(x_0)+\nabla u(x_0) \cdot (x-x_0) + \frac{a}{2}\, \vert x-x_0 \vert^2
$$
for $x$ in a neighborhood of $x_0$. 
In particular, 
$\nabla u (x) = \nabla u (x_0) + a(x-x_0)$ in such a
neighborhood, and hence the map 
$\nabla u  - a\iden$ is locally constant. 
Since $\Omega$ is connected we deduce that this map is indeed
a constant, say $\nabla u  - a\iden\equiv y_0$.

It follows that 
$\nabla u (\Gamma_u) = \nabla u (\Omega) = y_0 + a\Omega$.
By \eqref{gradmap} we know that $B_1(0)\subset \nabla u (\Gamma_u)= 
y_0 + a\Omega$. In addition,
these two open smooth sets, $B_1(0)$ and $y_0+a\Omega$, have the same measure since equality occurs in
the first inequality of \eqref{ineq}. 
We conclude that $B_1(0) = \nabla u (\Gamma_u)= 
y_0 + a\Omega$ and hence that $\Omega$ is a ball.
\end{proof}

The previous proof is also suited for a quantitative version as we will show in \cite{CCPRS} with Cinti,
Pratelli, Ros-Oton, and Serra.

\begin{rem}\label{proofABP}
The ABP estimate \eqref{abp} is proved proceeding as in the previous
proof for the isoperimetric
inequality, but now considering the Dirichlet problem \eqref{dirich} instead of 
\eqref{eqsem}. The main claim \eqref{gradmap} is now replaced by
$B_{M/d}(0)\subset\nabla u(\Gamma^{u})$, where
$M=\sup_\Omega u$, $d=\diam(\Omega)$
and $\Gamma^{u}$ is now the upper contact set of $u$. 
See Chapter~9 of \cite{GT} for details.
\end{rem}

In 1994 (before our proof), Trudinger \cite{T} had given a proof of the 
classical isoperimetric inequality
using the Monge-Amp\`ere operator and the ABP estimate. 
His proof consists of applying the ABP estimate to the problem
$$
\left\{
\alignedat2
\det D^2u &= \chi_\Omega &\quad &\text{in } B_R \\
u      &= 0         &\quad &\text{on }\partial B_R ,
\endalignedat
\right.
$$
where $\chi_\Omega$ is the characteristic 
function of $\Omega$ and $B_R=B_R(0)$, and then letting $R\to\infty$.

Before the proofs in \cite{T} and \cite{CSCM} using ABP, there was already Gromov's proof \cite{G}
of the isoperimetric inequality, which used the Knothe map (see also \cite{Cha} for a presentation).
A more classical proof of the isoperimetric problem is based on Steiner symmetrization; see \cite{F,O,Be}.
A fifth proof consists of deducing easily the isoperimetric inequality from the Brunn-Minkowski
inequality \eqref{brunn}; see \cite{Ga}.
Finally, in 2004 Cordero-Erausquin, Nazaret, and Villani \cite{CNV} used the Brenier map
from optimal transportation to give another proof of the isoperimetric inequality. This optimal transport proof,
as well as the Knothe-Gromov one, both lead also to the Wulff isoperimetric inequality for anisotropic perimeters
---which is discussed in the following section.

\section{The Wulff isoperimetric inequality}

In a personal communication, Robert McCann pointed out
that the previous proof also establishes the following
inequality concerning Wulff shapes and surface energies of crystals. 
Given any positive and smooth 
function $H$ on $\mathbb{S}^{n-1}=\partial B_1$
(the surface tension), consider the convex 
set $W\subset\R^n$ (called the Wulff shape) defined by
\begin{equation}\label{wulffshape}
W=\{ p\in\R^n \ :\ p\cdot\nu < H(\nu)\ \text{ for all }
\nu\in \mathbb{S}^{n-1}\}.
\end{equation}
Note that $W$ is an open set with $0\in W$.
To visualize $W$, it is useful to note that it is the intersection of the half-spaces
$\{p\cdot\nu<H(\nu)\}$ among all $\nu\in \mathbb{S}^{n-1}$. In particular, $W$ is a convex set.

For every smooth domain $\Omega\in\R^n$ (not necessarily convex),
define
$$
P_H(\Omega):=\int_{\partial\Omega} H(\nu(x))\, dS (x)
$$
to be its surface energy ---here $dS (x)$ denotes
the area element on $\partial\Omega$ and  $\nu(x)$ is the
unit exterior normal to $\partial\Omega$ at $x$. 
Then, among sets $\Omega$ with measure 
$\vert W\vert$,
the surface energy $P_H(\Omega)$ is minimized by (and only by) 
the Wulff shape $W$ and its translates.
Equivalently, for every $\Omega$ (without restriction on its
measure) we have:

\begin{theorem}[\cite{W,T1,T2}]
Let $\Omega$ be a bounded smooth domain of $\R^n$. Then
$$
\frac{P_H(\Omega)}{\vert \Omega \vert^{\frac{n-1}{n}}}
\ge \frac{P_H(W)}{\vert W \vert^{\frac{n-1}{n}}} ,
$$
with equality if only if $\Omega=aW+b$ for some $a>0$
and $b\in\R^n$. 
\end{theorem}

This theorem was first stated, without proof, by Wulff \cite{W} in 1901.
His work was followed by Dinghas \cite{D}, who studied the problem within the class of convex polyhedra.
He used the Brunn-Minkowski inequality
\begin{equation}\label{brunn}
 |A+B|^{\frac1n}\geq |A|^{\frac1n}+|B|^{\frac1n},
\end{equation}
valid for all nonempty measurable sets $A$ and $B$ of $\mathbb R^n$ for which $A+B$ is also measurable; 
see \cite{Ga} for more information on this inequality.
Some years later, Taylor \cite{T1,T2} finally proved the theorem  among sets of finite perimeter
---see \cite{CRS2} for more references in this subject. As mentioned in the previous section,
this anisotropic isoperimetric inequality also follows easily using the Knothe-Gromov map or the 
Brenier map from optimal transport. In addition, a proof of the Wulff theorem using an anisotropic rearrangement 
was given by Van Schaftingen (with a method coming from Klimov~\cite{Kl}).

This anisotropic isoperimetric problem can be solved with the same method
that we have used above for the isoperimetric problem. 
One considers now
the solution of 
\begin{equation*}
\left\{
\alignedat2
\Delta u &= \frac{P_H(\Omega)}{\vert \Omega \vert} 
&\quad &\text{in } \Omega\\
\frac{\partial u}{\partial\nu} &=H(\nu) &\quad &\text{on }\partial \Omega ,
\endalignedat
\right.
\label{eqsemW}
\end{equation*}
Claim \eqref{gradmap} is now replaced by
$W\subset\nabla u(\Gamma_u)$, which is proved again using
the Legendre transform of $u$. Then, the area formula
gives $\vert W\vert\le\{P_H(\Omega)/(n\vert\Omega\vert)\}^n
\vert\Omega\vert$. 

To conclude, one uses
that $P_H(W)=n\vert W\vert$. 
This last equality follows from the fact
that $H(\nu(p))=p\cdot\nu(p)$ for almost every $p\in\partial W$
(here $\nu(p)$ denotes the unit exterior normal to $\partial W$ at $p$), and thus
\begin{equation*}\label{per/vol}
P_{H}(W)=\int_{\partial W} H(\nu(x))dS=\int_{\partial W}x\cdot\nu(x)dS=\int_W \div(x)dx = n|W|,
\end{equation*}

A similar argument as in the previous section shows that equality is only achieved by
the sets $\Omega=aW+b$; see \cite{CRS2} for details.

\section{Weighted isoperimetric and Sobolev inequalities in convex cones}

The isoperimetric inequality in convex cones of Lions and Pacella reads as follows.

\begin{theorem}[\cite{LP}] \label{isopcon}
Let $\Sigma$ be an open convex cone in $\mathbb R^n$ with vertex at $0$, and $B_1:=B_1(0)$.
Then,
\begin{equation*}\label{pp}
\frac{P(\Omega;\Sigma)}{|\Omega\cap\Sigma|^{\frac{n-1}{n}} }\ge \frac{P(B_1;\Sigma)}{|B_1\cap\Sigma|^{\frac{n-1}{n}}}
\end{equation*}
for every measurable set $\Omega\subset \R^n$ with $|\Omega\cap\Sigma|<\infty$.
Here $P(\Omega;\Sigma)$ is the perimeter of $\Omega$ relative to $\Sigma$. It agrees with
the $(n-1)$-Hausdorff measure of $\partial \Omega \cap \Sigma$ for smooth sets $\Omega$.
\end{theorem}

Note that $\Sigma$ is an open set. Hence, if there is a  part of $\partial \Omega$ contained in $\partial\Sigma$, 
then it is not counted in this perimeter. The assumption of convexity of the cone can not be removed as shown in \cite{LP}. 

The proof of Theorem \ref{isopcon} given in \cite{LP} is based on the Brunn-Minkowski inequality
\eqref{brunn}. Alternatively, 
Theorem \ref{isopcon} can also be deduced from a degenerate case of the classical Wulff inequality 
of Section~4. For this, one must allow the surface energy $H$ to vanish in part of $\mathbb{S}^{n-1}$.
More precisely, we say that a function $H$ defined in $\R^n$ is a \emph{gauge} when
\begin{equation}\label{gauge}
H\textrm{ is nonnegative, positively homogeneous of degree one, and convex}.
\end{equation}
The Wulff inequality can be proved for such surface energies $H$. With this in hand, one can establish the Lions-Pacella
inequality as follows.

It is easy to prove that the convex set $B_1\cap\Sigma$ is equal to the Wulff shape $W$, defined by \eqref{wulffshape},
for a unique gauge $H$ (which depends on the cone $\Sigma$). 
This function $H$ vanishes on normal vectors to $\partial\Sigma$ and agrees with 1 on 
unit vectors inside $\Sigma$. This is why one can recover the Lions-Pacella inequality from the Wulff one associated
to this $H$. In particular, the Lions-Pacella inequality can be proved using the ABP method; see \cite{CRS2}
for more details.

Let us now turn to the extension of the Lions-Pacella theorem in \cite{CRS2} to the case 
of some homogeneous weights, as explained in the Introduction. Given a gauge $H$ and a nonnegative function
$w$ defined in $\overline \Sigma$, consider the weighted anisotropic perimeter
\begin{equation*} \label{defperiint}
P_{w,H}(\Omega;\Sigma) := \int_{\partial \Omega \cap \Sigma}H\bigl(\nu(x)\bigr)w(x)dS,
\end{equation*}
(defined in this way when $\partial\Omega$ is regular enough) and the weighted measure
$$
w(\Omega\cap\Sigma):=\int_{\Omega\cap\Sigma} w(x)\, dx.
$$

\begin{theorem}[\cite{CRS2}]\label{th1}
Let $H$ be a gauge in $\R^n$, i.e., a function satisfying \eqref{gauge}, and $W$ its associated Wulff shape defined by 
\eqref{wulffshape}.
Let $\Sigma$ be an open convex cone in $\mathbb R^n$ with vertex at the origin, and such that
$W\cap \Sigma\neq \varnothing$.
Let $w$ be a continuous function in $\overline\Sigma$, positive in $\Sigma$, and positively homogeneous of degree 
$\alpha\geq0$.
Assume in addition that $w^{1/\alpha}$ is concave in $\Sigma$ in case $\alpha>0$.

Then, for each measurable set $\Omega\subset\R^n$ with $w(\Omega\cap \Sigma)<\infty$,
\begin{equation}\label{mainresult}
\frac{P_{w,H}(\Omega;\Sigma) }{w(\Omega\cap \Sigma)^{\frac{D-1}{D}} }\geq \frac{P_{w,H}(W;\Sigma) }
{w(W\cap \Sigma)^{\frac{D-1}{D}}},
\end{equation}
where $D=n+\alpha$.
\end{theorem}

After announcing our result in \cite{CRS1} and posting the preprint \cite{CRS2}, 
E. Milman and L. Rotem \cite{MR} have found an alternative proof of our isoperimetric inequality,  Theorem \ref{th1}
(\cite{MR} mentions that the same has been found independently by Nguyen).
Their proof uses the Borell-Brascamp-Lieb extension of the the Brunn-Minkowski inequality.

Our key hypothesis that $w^{1/\alpha}$ is a concave function is equivalent to a
natural curvature-dimension bound, in fact to the nonnegativeness of a
Bakry-\'Emery Ricci tensor in dimension $D=n+\alpha$. This was pointed out by C. Villani.

Note that the shape of the minimizer is $W\cap\Sigma$, and that $W$ depends only on $H$ and not on the weight 
$w$ neither on the cone $\Sigma$.
In particular, in the isotropic case $H=\|\cdot\|_{2}$ we find the following noteworthy fact.
Even if the weights that we consider are not radial (unless $w\equiv \textrm{constant}$), 
still Euclidean balls centered at 
the origin (intersected with the cone) minimize this isoperimetric quotient.

Equality in \eqref{mainresult} holds whenever $\Omega\cap \Sigma=rW\cap \Sigma$, where $r$ is any positive number.
That $rW\cap\Sigma$ is the unique minimizer of \eqref{mainresult} will be shown in the upcoming paper \cite{CCPRS},
where in addition we show a quantitative version of \eqref{mainresult}.

Note also that we allow $w$ to vanish somewhere (or everywhere) on $\partial\Sigma$.
This happens in the case of the monomial weights \eqref{mon}, for which the previous
theorem holds. From \eqref{mainresult}, it is simple to deduce the sharp Sobolev inequality with monomial weights
\eqref{Sob} stated in the introduction.

Next, to show the key ideas in a simpler situation, 
we prove Theorem \ref{th1} in the isotropic case $H=\|\cdot\|_2$ when the weight $w\equiv0$ on 
$\partial\Sigma$. This is the case of the monomial weights. To simplify, we also assume that
$\Omega= U\cap \Sigma$, where $U$ is some bounded smooth domain in $\R^n$.

Let $w$ be a positive homogeneous function of degree $\alpha>0$ in an open convex cone $\Sigma\subset\R^n$.
In the proof we will need an easy lemma stating that $w^{1/\alpha}$ is concave in $\Sigma$ if and only if
\begin{equation}\label{lemmaw}
\alpha\left(\frac{w(z)}{w(x)}\right)^{1/\alpha}\leq \frac{\nabla w(x)\cdot z}{w(x)}
\end{equation}
holds for each $x,z\in\Sigma$; see \cite{CRS2}.

To prove the result we will also need the following equality. Here we denote $P_{w,H}$ by $P_w$
since $H$ is the Euclidean norm. Using that $w\equiv 0$ on $\partial\Sigma$, we deduce
\begin{equation}\label{formula-per=Dvol-for-Wulff}
\begin{split} P_w(W;\Sigma)&= \int_{\partial W\cap \Sigma} H(\nu(x))w(x)dS= \int_{\partial W\cap \Sigma} 
x\cdot\nu(x)\,w(x)dS\\
&=\int_{\partial(W\cap \Sigma)}x\cdot\nu(x)w(x)dS= \int_{W\cap\Sigma} \div(x w(x))dx\\
&=\int_{W\cap\Sigma} \left\{nw(x)+x\cdot\nabla w(x)\right\}dx
=\int_{W\cap\Sigma} (n+\alpha) w(x)  dx\\
&=D\,w(W\cap \Sigma),
\end{split}
\end{equation}
where we have used that $x\cdot \nabla w(x)=\alpha w(x)$ since $w$ is homogeneous of degree $\alpha$.

A key point in the following proof is that, when $\Omega=B_1\cap\Sigma$, the function $u(x)=|x|^2/2$ solves
$w^{-1}\div(w\nabla u)=b$ for some constant $b$, the normal derivative of $u$ on $\partial B_1 \cap\Sigma$
is identically one, and the normal derivative of $u$ on $\partial\Sigma\cap B_1$ is identically zero.

\begin{proof}[Proof of Theorem \ref{th1} in the case $w\equiv 0$ on $\partial\Sigma$ and $H=\|\cdot\|_2$]
For the sake of simplicity we assume here that $\Omega= U\cap \Sigma$, where $U$ is some bounded
smooth domain in $\R^n$.

Observe that since $\Omega= U\cap \Sigma$ is piecewise Lipschitz, and $w\equiv0$ on $\partial\Sigma$, it holds
\begin{equation}
\label{perimeteronlyw}
P_w(\Omega ; \Sigma) = \int_{\partial \Omega} w(x)dx.
\end{equation}
Hence, using that $w\in C(\overline \Sigma)$ and \eqref{perimeteronlyw}, it is immediate to prove that for any $y\in \Sigma$ 
we have
\[  \lim_{\delta \downarrow 0} P_w(\Omega + \delta y;\Sigma) = 
P_w(\Omega;\Sigma) \quad \mbox{and} \quad \lim_{\delta \downarrow 0} w(\Omega + \delta y) = w(\Omega).\]
We have denoted $\Omega + \delta y  = \{x + \delta y \, , \ x\in \Omega\}$.
Note that $P_w(\Omega + \delta y;\Sigma)$ could not converge to $P_w(\Omega;\Sigma)$ as 
$\delta \downarrow 0$ if $w$ did not vanish on the boundary of the cone $\Sigma$.

By this approximation property and a subsequent regularization of $\Omega + \delta y$ (a detailed argument can be found 
in \cite{CRS2}), we see
that it  suffices to prove \eqref{mainresult} for smooth domains whose closure is contained in $\Sigma$.
Thus, from now on in the proof, $\Omega$ is a smooth domain satisfying 
$\overline\Omega\subset \Sigma$.

At this stage, it is clear that by approximating $w|_{\overline{\Omega}}$ we can assume $w\in C^\infty(\overline\Omega)$
and $w>0$ in $\overline\Omega$.

Let $u$ be a solution of the linear Neumann problem
\begin{equation}
\left\{
\alignedat2
w^{-1}\textrm{div}(w\nabla u) &= b_\Omega
&\quad &\text{in } \Omega\\
\frac{\partial u}{\partial\nu} &=1 &\quad &\text{on }\partial \Omega .
\endalignedat
\right.
\label{eqsem2}
\end{equation}
The Fredholm alternative ensures that there exists a solution of \eqref{eqsem2} (which is unique up to an additive constant) if and only if the constant $b_\Omega$ is given by
\begin{equation}\label{cttb}
b_\Omega=\frac{P_w(\Omega;\Sigma)}{w(\Omega)}.\end{equation}
Note also that since $w$ is positive and smooth in $\overline\Omega$, \eqref{eqsem2} is a uniformly elliptic problem with 
smooth coefficients.
Thus, $u\in C^{\infty}(\overline\Omega)$.
For these classical facts, see Example 2 in Section 10.5 of \cite{H}, or the end of Section 6.7 of \cite{GT}.

Consider now the lower contact set of $u$, $\Gamma_u$, defined by \eqref{lcset} as the set of points in $\Omega$ at which the
tangent hyperplane to the graph of $u$ lies below $u$ in all $\overline \Omega$.
Then, as in Section~3, we touch by below the graph of $u$ with hyperplanes of fixed slope $p\in B_1$, 
and using the boundary condition in \eqref{eqsem2} we deduce that $B_1 \subset \nabla u (\Gamma_u)$.
From this, we obtain
\begin{equation*}\label{fivepointstar}
B_1\cap \Sigma\subset \nabla u(\Gamma_u)\cap \Sigma
\end{equation*}
and thus
\begin{equation}\label{ineqsec3}
\begin{split}
w(B_1\cap \Sigma) &\leq \int_{\nabla u (\Gamma_u)\cap\Sigma}w(p)dp \\
&\leq \int_{\Gamma_u\cap (\nabla u)^{-1}(\Sigma)} w(\nabla u(x))\det D^2u(x)\,dx\\
&\leq  \int_{\Gamma_u\cap (\nabla u)^{-1}(\Sigma)} w(\nabla u)\left(\frac{\Delta u}{n}\right)^ndx.
\end{split}
\end{equation}
We have applied the area formula to the smooth map $\nabla u : \Gamma_u \rightarrow \mathbb R^n$ and also the classical 
arithmetic-geometric means inequality ---all eigenvalues of $D^2u$ are nonnegative in $\Gamma_u$ by definition of this set.

Next we use that, when $\alpha>0$,
\[s^{\alpha}t^n\leq \left(\frac{\alpha s+nt}{\alpha+n}\right)^{\alpha+n}\ \ \textrm{for all }\ s>0\ \textrm{and}\ t>0,\]
which follows from the concavity of the logarithm function. Using also \eqref{lemmaw}, we find
\[\frac{w(\nabla u)}{w(x)}\left(\frac{\Delta u}{n}\right)^n\leq
\left(\frac{\alpha\left(\frac{w(\nabla u)}{w(x)}\right)^{1/\alpha}+\Delta u}{\alpha+n}\right)^{\alpha+n}\leq
\left(\frac{\frac{\nabla w(x)\cdot \nabla u}{w(x)}+\Delta u}{D}\right)^{D}.\]
Recall that $D=n+\alpha$.
Thus, using the equation in \eqref{eqsem2}, we obtain
\begin{equation}\label{36}
\frac{w(\nabla u)}{w(x)}\left(\frac{\Delta u}{n}\right)^n\leq \left(\frac{b_\Omega}{D}\right)^{D}\ \ {\rm in}\ \Gamma_u\cap (\nabla u)^{-1}(\Sigma).
\end{equation}
If $\alpha=0$ then $w\equiv1$, and \eqref{36} is trivial.

Therefore, since $\Gamma_u\subset \Omega$, combining \eqref{ineqsec3} and \eqref{36} we obtain
\begin{equation*}\label{7}\begin{split}
w(B_1\cap\Sigma) &\leq \int_{\Gamma_u\cap (\nabla u)^{-1}(\Sigma)} \left(\frac{b_\Omega}{D}\right)^{D}w(x)dx=
\left(\frac{b_\Omega}{D}\right)^{D}w(\Gamma_u\cap (\nabla u)^{-1}(\Sigma))\\
&\leq \left(\frac{b_\Omega}{D}\right)^{D}w(\Omega)
= D^{-D}\frac{P_w(\Omega;\Sigma)^{D}}{w(\Omega)^{D-1}}.\end{split}
\end{equation*}
In the last equality we have used the value of the constant $b_\Omega$, given by \eqref{cttb}.

Finally, using that, by \eqref{formula-per=Dvol-for-Wulff}, we have $P_w(B_1\cap\Sigma;\Sigma) = D\,w(B_1\cap \Sigma)$,
we obtain the desired inequality \eqref{mainresult}.
\end{proof}

\end{document}